\documentclass[11pt]{article}
\usepackage[a4paper]{anysize}\marginsize{3.5cm}{3.5cm}{1.3cm}{2cm}
\pdfpagewidth=\paperwidth \pdfpageheight=\paperheight
\usepackage{amsfonts,amssymb,amsthm,amsmath,eucal}
\usepackage{pgf}
\usepackage{tabularx, booktabs}
\usepackage{bbm,array}
\usepackage{tikz} 
\usepackage{float}
\usetikzlibrary{arrows}
\theoremstyle{plain}
\newtheorem{thm}{Theorem}[section]
\newtheorem{theorem}[thm]{Theorem}

\newtheorem{lemma}[thm]{Lemma}

\theoremstyle{definition}
\newtheorem{definition}[thm]{Definition}

\newtheorem{example}[thm]{Example}

\newtheorem{thevarthm}[thm]{\varthmname}

\newenvironment{varthm*}[1]{\trivlist\item[]{\bf #1.}\it}{\endtrivlist}

\renewcommand\geq{\geqslant}

\renewcommand\leq{\leqslant}

\newcommand\be{\begin{eqnarray*}}
\newcommand\ee{\end{eqnarray*}}

\newcommand\F{\mathbb F}

\renewcommand\P{\mathbb P}

\newcommand\call{{\mathcal L}}

\newcommand\newop[2]{\def#1{\mathop{\rm #2}\nolimits}}
\newop\log{log}
\newop\ord{ord}
\newop\Gal{Gal}
\newop\SL{SL}
\newop\GL{GL}
\newop\Bl{Bl}
\newop\mult{mult}
\newop\mass{mass}
\newop\div{div}
\newop\codim{codim}
\newop\sing{sing}
\newop\vdim{vdim}
\newop\edim{edim}
\newop\Ass{Ass}
\newop\size{size}
\newop\reg{reg}
\newop\areg{areg}
\newop\asreg{asreg}
\newop\satdeg{satdeg}
\newop\supp{supp}
\newop\gin{gin}
\newop\ini{in}
\newop\vol{vol}
\newop\sat{sat}
\newop\chara{char}
\newop\length{length}
\newop\depth{depth}
\newop\characteristic{char}
\newcommand\eqnref[1]{(\ref{#1})}

\newcolumntype{Y}{>{\centering\arraybackslash}X}

\def\keywordname{{\bfseries Keywords}}%
\def\keywords#1{\par\addvspace\medskipamount{\rightskip=0pt plus1cm
\def\and{\ifhmode\unskip\nobreak\fi\ $\cdot$
}\noindent\keywordname\enspace\ignorespaces#1\par}}
\def\subclassname{{\bfseries Mathematics Subject Classification
(2000)}\enspace}
\def\subclass#1{\par\addvspace\medskipamount{\rightskip=0pt plus1cm
\def\and{\ifhmode\unskip\nobreak\fi\ $\cdot$
}\noindent\subclassname\ignorespaces#1\par}}

\definecolor{qqqqff}{rgb}{0,0,1}
\definecolor{uuuuuu}{rgb}{0.27,0.27,0.27}
\definecolor{zzttqq}{rgb}{0.6,0.2,0}
\definecolor{xdxdff}{rgb}{0.49019607843137253,0.49019607843137253,1.0}

\begin{document}

\author{Justyna~Szpond}
\title{On linear Harbourne constants}
\date{\today}
\maketitle
\thispagestyle{empty}

\begin{abstract} In this note we compute values of global linear Harbourne constants over arbitrary fields
for up to ten lines. These invariants have appeared recently in the discussions around the Bounded
Negativity Conjecture, see \cite{BNC}. They seem to be of independent interest also from the point
of view of combinatorics.

\keywords{ Bounded Negativity Conjecture, arrangements of lines, combinatorial arrangements}
\subclass{14C20, 52C30, 05B30}
\end{abstract}


\section{Introduction}
\label{intro}
In recent years, there has been growing interest in negative curves on algebraic surfaces. The Bounded Negativity Conjecture (BNC for short) is probably the most interesting open question in this area. The BNC predicts that for any smooth complex surface $X$ there exists a lower bound for the selfintersection of reduced divisors on $X$. It is not known if the Bounded Negativity property is invariant in the birational class of a surface, i.e. given birational surfaces $X$ and $Y$, it is not known if curves on $X$ have bounded negativity if and only if they do on $Y$. As the first step towards understanding this question, in \cite{BNC} the authors introduce and study Harbourne constants (see Definition \ref{Hconstants} for details). 
It is well known that these constants can become arbitrarily small in finite characteristic
when allowing $d$ to grow infinitely. It is an intriguing question if they are bounded over the field of complex numbers.
The purpose of this note is to actually compute these constants for a low number of lines (up to $10$) defined over an \emph{arbitrary} field. This is a problem of combinatorial flavor and we hope that results presented here might be of interest also from this point of view. We restrict to up to $10$ lines since configurations of more lines cannot be easily dealt with
the tools developed here, see \cite{LATR}. We hope to come back to this case in the sequel paper.\\
Our main result is Theorem \ref{main}.

   \section{Configurations of lines}
In this section we collect some basic facts about configurations of lines in projective planes. By a configuration of lines we understand a family $\mathcal{L}=\{L_1,\ldots,L_d\}$ of mutually distinct lines. Let $\mathcal{P}(\mathcal{L})=\{P_1,\ldots,P_s\}$ be the set of singular points of the configuration i.e. intersection points of lines in $\mathcal{L}$. For a point $P\in\mathbb{P}^2$ we denote by $m_{\mathcal{L}}(P)$ the number of lines from $\mathcal{L}$ passing through that point. Let $t_k=t_k(\mathcal{L})$ denote the number of points where exactly $k\geq 2$ lines meet. Then there is the obvious combinatorial equality
   \begin{equation}\label{eq: combinatorial}
      \binom{d}{2}=\sum\limits_{k\geq 2}t_k\binom{k}{2}.
   \end{equation}
This equality provides the first constrain on the existence of line configuration. Let $T=(t_2,\ldots,t_d)$ denote a non-negative integral solution to \eqref{eq: combinatorial}. It is not clear when such a solution comes from a configuration of lines. Therefore we present some more geometrical criteria for the existence of certain line configurations. They are surely well know to experts but we were not able to find a proper citation in the literature. Thus we collect these facts for the convenience of a reader.
\begin{lemma}[Parity Criterion]\label{lem: points on one line}
   Let $\call=\left\{L_1,\ldots,L_d\right\}$ be a configuration of lines.
   Let $L\in\call$ be a fixed line. Let $P_1(L),\ldots,P_r(L)$
   be intersection points of $L$ with other configuration lines
   with multiplicities $m_1(L),\ldots,m_r(L)$ respectively. Then
   $$d-1=\sum\limits_{k=1}^r(m_k(L)-1).$$
\end{lemma}
\begin{lemma}
Let $\mathcal{L}=\{L_1,\ldots,L_d\}$ be a configuration of lines, let  $\mathcal{P}(\mathcal{L})=\{P_1,\ldots,P_s\}$, be the set of singular points of the configuration. Let $m_1,\ldots,m_s$ be the multiplicities of points in $\mathcal{P}(\mathcal{L})$. Without loss of generality we can assume that
\[
m_1\geq m_2\geq\ldots\geq m_s.
\]
Then there are inequalities
\[
m_1+\ldots +m_r\leq d+\binom{r}{2}
\]
for $r=1,\ldots, s$. In the sequel we shall make frequent use of the following two particular cases:
\begin{description}
  \item[a)]\label{lem: multiplicity} for $s\geq 3$ there is a Triangular Inequality
  $$m_1+m_2+m_3\leq d+3,$$
  \item[b)]\label{lem: multiplicity4} for $s\geq 4$ there is a Quadrangle Inequality
  $$m_1+m_2+m_3+m_4\leq d+6.$$
\end{description}
\end{lemma}
\begin{proof}
We count lines. For $r=1$ we have obviously $m_1\leq d$.
For $r=2$ we have $m_1+m_2\leq d+1$ since at most one line is counted twice. For $r=3$ we have $m_1+m_2+m_3\leq d+3$ since at most three lines are counted twice (lines through pairs of points $P_1$, $P_2$, $P_3$). The general case follows in the same manner.
\end{proof}
We omit the proofs of the following simple facts
\begin{lemma}[Two Pencils Inequality]\label{boundary_of_points}
Let $\mathcal{L}=\{L_1,\ldots,L_d\}$ be a configuration of lines, let  $\{P_1,\ldots,P_s\}$, where $s\geq 2$, be the set of singular points of the configuration. Let $m_1,\ldots,m_s$ be the multiplicities of points $P_1,\ldots,P_s$ respectively. Without loss of generality we can assume that
\[
m_1\geq m_2\geq\ldots\geq m_s.
\]
\begin{description}
\item[a)] If points $P_1$ and $P_2$ do not lie on a configuration line, then $$m_1m_2+2\leq s,$$
\item[b)] If points $P_1$ and $P_2$ lie on a configuration line, then $$(m_1-1)(m_2-1)+2\leq s.$$
\end{description}
\end{lemma}
\begin{lemma}[The Pencil Criterion]\label{SPC}
If the union of lines in pencils through points $P_1,\ldots, P_r$ gives the whole configuration, then all configuration points (apart from $P_1,\ldots, P_r$) have multiplicity at most $r$.
\end{lemma}
For configurations of lines in $\mathbb{P}^2(\mathbb{C})$ the following inequality due to Hirzebruch is extremely useful, see \cite{H}.
\begin{theorem}\label{hirzebruch}
Let $\mathcal{L}$ be a configuration of $d$ lines in the complex projective plane $\mathbb{P}^2$. Then
\[
t_2+\frac{3}{4}t_3\geq d+\sum_{k\geq 5}(k-4)t_k
\]
provided $t_d=t_{d-1}=0$.
\end{theorem}
\section{Harbourne constants}
Let $\mathbb{K}$ be an arbitrary field.
\begin{definition}\label{Hconstants}
Let $\mathcal{L}=\{L_1,\ldots,L_d\}$ be a configuration of lines in the projective plane $\mathbb{P}^2(\mathbb{K})$, let  $\mathcal{P}(\mathcal{L})=\{P_1,\ldots,P_s\}$ be the set of all singular points of the configuration. Then the \emph{linear Harbourne constant of $\mathcal{L}$ at $\mathcal{P}$} is defined as
\begin{equation}\label{l-H-constant}
H_{L}(\mathbb{K},\mathcal{L})=\frac{d^2-\sum_{k=1}^s m_{\mathcal{L}}(P_k)^2}{s}
\end{equation}
Similarly, we define the \emph{linear Harbourne constant of configurations of $d$ $\mathbb{K}$-lines} as the minimum
\[
H_{L}(\mathbb{K},d):=\min H_{L}(\mathbb{K},\mathcal{L})
\]
where the minimum is taken over all configurations of $d$ lines in $\mathbb{P}^2(\mathbb{K})$.\\
Going over all fields $\mathbb{K}$, we introduce the \emph{absolute linear Harbourne constant} as
\[
H_{L}(d):=\min_{\mathbb{K}} H_{L}(\mathbb{K},d).
\]
\end{definition}
Motivated by \eqref{l-H-constant}, it is convenient for a solution $T=(t_2,\ldots,t_d)$ of equation \eqref{eq: combinatorial}, to define the combinatorial quotient $q(T)$ associated to $T$ as
\begin{equation}\label{combin-equotient}
q(T)=\frac{d^2-\sum_{k=2}^d k^2t_k}{\sum_{k=2}^d t_k}.
 \end{equation}
Note that if there exists a configuration of lines $\mathcal{L}$ with $T=(t_2(\mathcal{L}),\ldots, t_d(\mathcal{L}))$, then we have $q(T)=H_L(\mathbb{K},\mathcal{L})$. The point is that not every solution of \eqref{eq: combinatorial} can be realized in the geometrical way.
\begin{example}\label{example}
Let $\mathbb{K}$ be a field.
\begin{enumerate}
\item In the case when $d$ lines meet in one point (i.e. they belong to the same pencil) we have
\[
H_{L}(\mathbb{K},\mathcal{L})=0.
\]
\item For a configuration of $d$ general lines we have
\[
H_{L}(\mathbb{K},\mathcal{L})=-2+\frac{2}{d-1},
\]
hence in this case the constant is always greater then $-2$.
\end{enumerate}
\end{example}
As the main result of this note we establish the values of the absolute linear Harbourne constants of up to ten lines. Since the case $\mathbb{K}=\mathbb{C}$ is of particular interest from the point of view of BNC, we compute separately also the complex Harbourne constants.
\begin{theorem}\label{main}
The values of the absolute Harbourne constants are
\begin{center}
\begin{tabularx}{1.05\textwidth}{|c|Y|Y|Y|Y|Y|Y|Y|Y|Y|}
\hline
   $d$ & $2$ & $3$ & $4$ & $5$ & $6$ & $7$ & $8$ & $9$ & $10$\rule{0pt}{2.6ex}\\
\hline
   $H_{L}(d)$ & $0$ & $-1$ & $-1\frac{1}{3}$ & $-1,5$ & $-1\frac{5}{7}$ & $\mathbf{-2}$ & $-2$ & $-2,25$ & $\mathbf{-2\frac{5}{12}}$\rule{0pt}{2.6ex}\\
\hline
\end{tabularx}
\end{center}
Over $\mathbb{C}$ we obtain the following values
   \begin{center}
   \begin{tabularx}{1.05\textwidth}{|c|Y|Y|Y|Y|Y|Y|Y|Y|Y|}
   \hline
   $d$ & $2$ & $3$ & $4$ & $5$ & $6$ & $7$ & $8$ & $9$ & $10$\rule{0pt}{2.6ex}\\
  \hline
   $H_{L}(\mathbb{C},d)$ & $0$ & $-1$ & $-1\frac{1}{3}$ & $-1,5$ & $-1\frac{5}{7}$ & $\mathbf{-1\frac{8}{9}}$ & $-2$ & $-2,25$ & $\mathbf{-2\frac{4}{15}}$\rule{0pt}{2.6ex}  \\
   \hline
   \end{tabularx}
   \end{center}
\end{theorem}

\section{Proof of the main Theorem}

\begin{proof}
Our approach is the following. We find the list of all solutions $T=(t_2,\ldots,t_d)$ for equation \eqref{eq: combinatorial}. Then we compute the corresponding quotients $q(T)$. Then we take the minimal obtained quotients and discuss if there exists a geometric configuration with invariants $(t_2,\ldots,t_d)$ either over $\mathbb{C}$ or over an arbitrary field.\\
Since any two lines belong to one pencil,
\[
H_{L}(2)=0,
\]
by Example \ref{example}.\\
Three lines either meet in one triple point (see Figure~\ref{fig: d=3s}), or are general (see Figure~\ref{fig: d=3}). So the minimal value of the Harbourne's constant is $-1$.
\begin{figure}[H]
\centering
\begin{minipage}{0.4\textwidth}
   \centering
\begin{tikzpicture}[line cap=round,line join=round,>=triangle 45,x=1.0cm,y=1.0cm,scale=0.5]
\clip(-4.3,-1.92) rectangle (7.06,6.3);
\draw [domain=-4.3:7.06] plot(\x,{(--9.5112--1.5*\x)/4.56});
\draw [domain=-4.3:7.06] plot(\x,{(--13.8448-2.32*\x)/5.08});
\draw [domain=-4.3:7.06] plot(\x,{(--1.09380350785--0.986424897798*\x)/0.805931689305});
\end{tikzpicture}
 \caption{$ $  $d=3, t_2=0,t_3=1$}\label{fig: d=3s}
   \end{minipage}
   \quad
   \begin{minipage}{0.4\textwidth}
   \centering
\begin{tikzpicture}[line cap=round,line join=round,>=triangle 45,x=1.0cm,y=1.0cm,scale=0.5]
\clip(-5.22,-0.7199999999999988) rectangle (6.140000000000001,7.5);
\draw [domain=-5.22:6.140000000000001] plot(\x,{(--10.8936--0.040000000000000036*\x)/6.279999999999999});
\draw [domain=-5.22:6.140000000000001] plot(\x,{(--12.8432--3.5200000000000005*\x)/2.76});
\draw [domain=-5.22:6.140000000000001] plot(\x,{(--20.0456-3.4800000000000004*\x)/3.52});
\begin{scriptsize}
\draw [fill=black] (-2.3,1.72) circle (1.5pt);
\draw [fill=black] (3.98,1.76) circle (1.5pt);
\draw [fill=black] (0.46,5.24) circle (1.5pt);
\end{scriptsize}
\end{tikzpicture}
\caption{$ $  $d=3, t_2=3, t_3=0$}\label{fig: d=3}
\end{minipage}
\end{figure}

If we have $4$ lines there are four solutions for equation \eqref{eq: combinatorial}
\begin{enumerate}
\item[\rm(i)] $T=(6,0,0)$, and $q(T)=-1 \frac13$,
\item[\rm(ii)] $T=(3,1,0)$, and $q(T)=-1 \frac14$,
\item[\rm(iii)] $T=(0,2,0)$, and $q(T)=-1$,
\item[\rm(iv)] $T=(0,0,1)$, and $q(T)=0$.
\end{enumerate}
Case (iii) is excluded by Lemma \ref{lem: points on one line}. Cases (i), (ii) and (iv) are geometrically realizable. The minimal value of the constant is obtained for a configuration of general lines, see Figure~\ref{fig: d=4}.
\begin{figure}[H]
\centering
\begin{tikzpicture}[line cap=round,line join=round,>=triangle 45,x=1.0cm,y=1.0cm,scale=0.5]
\clip(-4.54,-0.7999999999999987) rectangle (6.82,7.42);
\draw [domain=-4.54:6.82] plot(\x,{(--12.8112-2.04*\x)/4.08});
\draw [domain=-4.54:6.82] plot(\x,{(--9.7696--1.1600000000000001*\x)/2.24});
\draw [domain=-4.54:6.82] plot(\x,{(--12.344-3.2*\x)/1.8399999999999999});
\draw [domain=-4.54:6.82] plot(\x,{(-1.1989213175760627-3.2422006537591153*\x)/-1.7951460900176015});
\begin{scriptsize}
\draw [fill=qqqqff] (-1.2,3.74) circle (1.5pt);
\draw[color=qqqqff] (-1.06,4.0200000000000005);
\draw [fill=qqqqff] (2.88,1.7) circle (1.5pt);
\draw[color=qqqqff] (3.02,1.9800000000000009);
\draw[color=black] (-4.380000000000001,5.140000000000001);
\draw [fill=qqqqff] (1.04,4.9) circle (1.5pt);
\draw[color=qqqqff] (1.1800000000000002,5.180000000000001);
\draw[color=black] (-4.380000000000001,2.440000000000001);
\draw[color=black] (-0.08000000000000002,7.28); 
\draw [fill=xdxdff] (2.8671460900176013,5.846200653759115) circle (1.5pt);
\draw[color=xdxdff] (3.0,6.12); 
\draw [fill=xdxdff] (1.0719999999999998,2.6039999999999996) circle (1.5pt);
\draw[color=xdxdff] (1.22,2.880000000000001); 
\draw[color=black] (3.3600000000000003,7.28); 
\draw [fill=uuuuuu] (1.7039341825020893,3.745331856518105) circle (1.5pt);
\draw[color=uuuuuu] (1.84,4.0200000000000005); 
\end{scriptsize}
\end{tikzpicture}
\caption{$ $  $d=4$}\label{fig: d=4}
\end{figure}

In the case of $d=5$ we have the following candidates for the minimal value of the Harbourne constant
\begin{enumerate}
\item[\rm(i)] $T=(10,0,0,0)$, and $q(T)=-1,5$,
\item[\rm(ii)] $T=(7,1,0,0)$, and $q(T)=-1,5$,
\item[\rm(iii)] $T=(4,2,0,0)$, and $q(T)=-1,5$,
\item[\rm(iv)] $T=(1,3,0,0)$, and $q(T)=-1,5$.
\end{enumerate}
Case (iv) is excluded by Lemma \ref{lem: multiplicity}. Cases (i), (ii) and (iii) are realizable,  see Figure~\ref{fig: d=5}.
\begin{figure}[H]
\centering
\begin{minipage}{0.4\textwidth}
   \centering
\begin{tikzpicture}[line cap=round,line join=round,>=triangle 45,x=1.0cm,y=1.0cm,scale=0.5]
\clip(-4.82,-1.6599999999999986) rectangle (6.54,6.5600000000000005);
\draw [domain=-4.82:6.54] plot(\x,{(--6.9376--0.5*\x)/2.06});
\draw [domain=-4.82:6.54] plot(\x,{(-2.144-1.46*\x)/-0.78});
\draw [domain=-4.82:6.54] plot(\x,{(--2.176-0.96*\x)/1.2799999999999998});
\draw [domain=-4.82:6.54] plot(\x,{(--6.3134503743148-4.325271730618638*\x)/0.24408770555990578});
\draw [domain=-4.82:6.54] plot(\x,{(-6.539890561042947-3.27805022397468*\x)/-3.934211057466365});
\begin{scriptsize}
\draw [fill=black] (-1.68,2.96) circle (1.5pt);
\draw[color=black] (-1.66,3.420000000000001); 
\draw [fill=black] (0.38,3.46) circle (1.5pt);
\draw[color=black] (0.16,3.8200000000000003); 
\draw[color=black] (-4.62,2.620000000000001); 
\draw [fill=black] (-0.4,2.0) circle (1.5pt);
\draw[color=black] (-0.12,2.2600000000000007);
\draw[color=black] (1.7000000000000002,6.36);
\draw[color=black] (-4.62,5.0600000000000005); 
\draw [fill=black] (1.179912294440094,4.957271730618638) circle (1.5pt);
\draw[color=black] (1.4400000000000002,5.12); 
\draw [fill=black] (1.4239999999999997,0.6319999999999997) circle (1.5pt);
\draw[color=black] (1.6,0.980000000000001); 
\draw[color=black] (0.86,6.4399999999999995); 
\draw [fill=black] (2.8881609400035613,4.06877692718533) circle (1.5pt);
\draw[color=black] (2.7600000000000002,4.5); 
\draw [fill=black] (-1.0460501174628036,0.79072670321065) circle (1.5pt);
\draw[color=black] (-0.8200000000000001,0.860000000000001); 
\draw[color=black] (-3.48,-1.3399999999999987); 
\draw [fill=black] (1.2524573879593537,3.6717615019318814) circle (1.5pt);
\draw[color=black] (1.4200000000000002,4.140000000000001); 
\draw [fill=black] (1.3045169330401387,2.7492583462310622) circle (1.5pt);
\draw[color=black] (1.6,2.8400000000000007); 
\draw [fill=black] (0.023804009489837942,1.6821469928826218) circle (1.5pt);
\draw[color=black] (0.1,1.620000000000001); 
\end{scriptsize}
\end{tikzpicture}
 \end{minipage}
   \quad
   \begin{minipage}{0.4\textwidth}
   \centering
  \begin{tikzpicture}[line cap=round,line join=round,>=triangle 45,x=1.0cm,y=1.0cm,scale=0.5]
\clip(-4.3,-1.9199999999999986) rectangle (7.0600000000000005,6.3);
\draw [domain=-4.3:7.0600000000000005] plot(\x,{(--10.016--1.6*\x)/3.0});
\draw [domain=-4.3:7.0600000000000005] plot(\x,{(--4.774400000000001-1.9400000000000002*\x)/0.9600000000000001});
\draw [domain=-4.3:7.0600000000000005] plot(\x,{(--5.888701989242723-1.5015487065653548*\x)/4.534786988815846});
\draw [domain=-4.3:7.0600000000000005] plot(\x,{(--2.581050519031143-3.2593079584775086*\x)/-1.5137024221453284});
\draw [domain=-4.3:7.0600000000000005] plot(\x,{(--12.94--0.8999999999999999*\x)/5.199999999999999});
\begin{scriptsize}
\draw [fill=black] (-2.36,2.08) circle (1.5pt);
\draw [fill=black] (1.6,1.74) circle (1.5pt);
\draw [fill=black] (0.64,3.68) circle (1.5pt);
\draw [fill=black] (2.1747869888158458,0.5784512934346454) circle (1.5pt);
\draw [fill=black] (3.1137024221453284,4.999307958477509) circle (1.5pt);
\draw [fill=black] (1.2090574675932202,0.8982215311680101) circle (1.5pt);
\draw [fill=black] (1.1326223520818117,2.684492330168006) circle (1.5pt);
\draw [fill=black] (2.117838154647017,2.8550104498427538) circle (1.5pt);
\end{scriptsize}
\end{tikzpicture}
\end{minipage}
\quad
   \begin{minipage}{0.4\textwidth}
   \centering
\begin{tikzpicture}[line cap=round,line join=round,>=triangle 45,x=1.0cm,y=1.0cm,scale=0.5]
\clip(-4.3,-1.9199999999999986) rectangle (7.0600000000000005,6.3);
\draw [domain=-4.3:7.0600000000000005] plot(\x,{(--11.0656--0.16000000000000014*\x)/4.46});
\draw [domain=-4.3:7.0600000000000005] plot(\x,{(--11.9848--1.48*\x)/3.5999999999999996});
\draw [domain=-4.3:7.0600000000000005] plot(\x,{(--5.9079999999999995-1.4*\x)/3.78});
\draw [domain=-4.3:7.0600000000000005] plot(\x,{(-5.105599999999999--1.3199999999999998*\x)/-0.8600000000000001});
\draw [domain=-4.3:7.0600000000000005] plot(\x,{(-1.6911999999999998--1.56*\x)/0.6800000000000002});
\begin{scriptsize}
\draw [fill=black] (-2.26,2.4) circle (1.5pt);
\draw [fill=black] (2.2,2.56) circle (1.5pt);
\draw [fill=black] (1.34,3.88) circle (1.5pt);
\draw [fill=black] (1.52,1.0) circle (1.5pt);
\draw [fill=uuuuuu] (3.0887677889621665,4.5989378687955575) circle (1.5pt);
\draw [fill=uuuuuu] (3.755887573964496,0.17189349112426053) circle (1.5pt);
\end{scriptsize}
\end{tikzpicture}
\end{minipage}
\caption{$ $  $d=5$ (configurations computing $H_L(5)$)}\label{fig: d=5}
\end{figure}
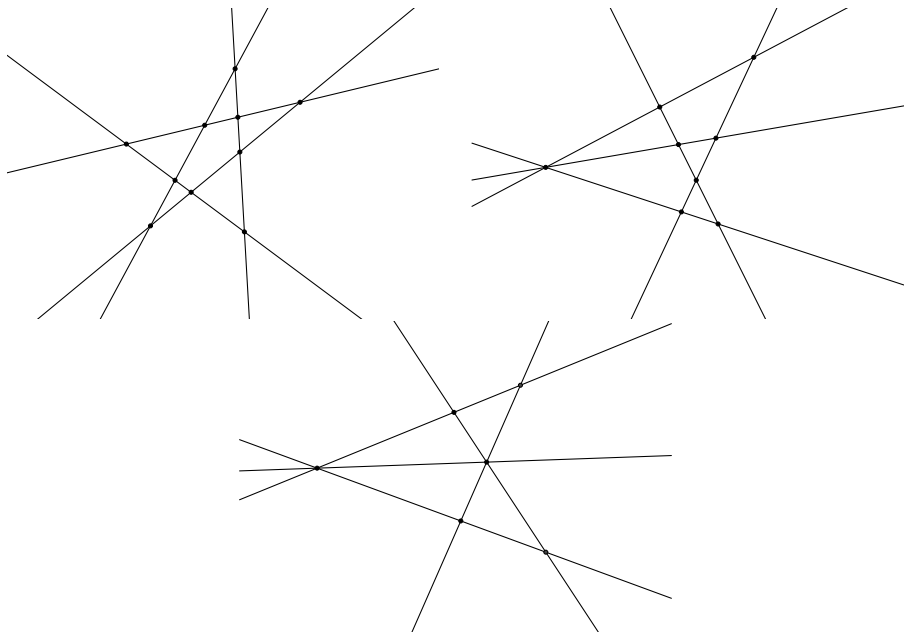

For $d=6$ combinatorial equality \eqref{eq: combinatorial} yields that there are the following candidates for the lowest value of Harbourne constant
\begin{enumerate}
\item[\rm(i)] $T=(0,5,0,0,0)$, and $q(T)=-1,8$,
\item[\rm(ii)] $T=(0,3,1,0,0)$, and $q(T)=-1,75$,
\item[\rm(iii)] $T=(3,4,0,0,0)$, and $q(T)=-1\frac{5}{7}\simeq -1,714$.
\end{enumerate}
Case (i) is excluded by Lemma \ref{lem: points on one line}. Passing to (ii) we observe that the $4$-fold point must be contained in every line (see Lemma \ref{lem: points on one line}), which is impossible. Then the minimal value of the Harbourne constant is $-1\frac{5}{7}$, see Figure~\ref{fig: d=6}.

\begin{figure}[H]
\centering
\begin{tikzpicture}[line cap=round,line join=round,>=triangle 45,x=1.0cm,y=1.0cm,scale=0.5]
\clip(-4.3,-1.92) rectangle (7.0600000000000005,6.3);
\draw [domain=-4.3:7.0600000000000005] plot(\x,{(--8.684800000000001--0.040000000000000036*\x)/3.4800000000000004});
\draw [domain=-4.3:7.0600000000000005] plot(\x,{(--8.0032--1.1800000000000002*\x)/2.58});
\draw [domain=-4.3:7.0600000000000005] plot(\x,{(--4.752-1.32*\x)/2.64});
\draw [domain=-4.3:7.0600000000000005] plot(\x,{(--4.684800000000001-1.1400000000000001*\x)/0.9000000000000001});
\draw [domain=-4.3:7.0600000000000005] plot(\x,{(-0.7664--1.36*\x)/0.8400000000000001});
\draw [domain=-4.3:7.0600000000000005] plot(\x,{(--3.2696000000000005-2.5*\x)/0.06000000000000005});
\begin{scriptsize}
\draw [fill=black] (-1.36,2.48) circle (1.5pt);
\draw [fill=black] (2.12,2.52) circle (1.5pt);
\draw [fill=black] (1.22,3.66) circle (1.5pt);
\draw [fill=black] (1.28,1.16) circle (1.5pt);
\draw [fill=black] (3.455672068636797,4.682516682554813) circle (1.5pt);
\draw [fill=black] (4.441739130434785,-0.4208695652173919) circle (1.5pt);
\draw [fill=black] (1.2476006618863762,2.5099724214009926) circle (1.5pt);
\end{scriptsize}
\end{tikzpicture}
\caption{$ $  $d=6$}\label{fig: d=6}
\end{figure}
\subsection{Seven lines}
 The combinatorial equality \eqnref{eq: combinatorial}
   implies that for $d=7$ we consider the following two cases for minimal Harbourne constant and the following cases need to be considered:
\begin{enumerate}
\item[\rm(i)] $T=(0,1,3,0,0,0)$ and $q(T)=-2$,
\item[\rm(ii)] $T=(0,3,2,0,0,0)$ and $q(T)=-2$,
\item[\rm(iii)] $T=(0,5,1,0,0,0)$ and $q(T)=-2$,
\item[\rm(iv)] $T=(0,7,0,0,0,0)$, and $q(T)=-2$,
\item[\rm(v)] $T=(3,6,0,0,0,0)$ and $q(T)=-1\frac{8}{9}$.
\end{enumerate}
Cases (i), (ii) and (iii) are excluded since the Parity Criterion (Lemma \ref{lem: points on one line}) yields that each configuration line must pass through an even number of $4$-fold points (we do not have double points in these configurations). It is well known that the configuration of $7$ lines with $7$ triple points exists only in characteristic $2$ (the smallest such configuration is the Fano plane $\P^2(\F_2)$) (see \cite{BGS74}). The picture below (Figure \ref{fig: d=7,Fano})
indicates collinear points as lying on the segments or on the circle (i.e. the circle is the seventh line).
\begin{figure}[H]
\centering
\begin{minipage}{0.4\textwidth}
\centering
\begin{tikzpicture}[line cap=round,line join=round,x=1.0cm,y=1.0cm,scale=0.7]
\clip(0.02,-2.74) rectangle (8,3.34);
\draw [color=uuuuuu] (1.74,-1.7)-- (6.72,-1.7);
\draw [color=uuuuuu] (6.72,-1.7)-- (4.23,2.61);
\draw
 [color=uuuuuu] (4.23,2.61)-- (1.74,-1.7);
\draw(4.23,-0.26) circle (1.44cm);
\draw (1.74,-1.7)-- (5.48,0.46);
\draw (4.23,2.61)-- (4.23,-1.7);
\draw (2.99,0.46)-- (6.72,-1.7);
\begin{scriptsize}
\fill [color=uuuuuu] (1.74,-1.7) circle (1.8pt);
\fill [color=uuuuuu] (6.72,-1.7) circle (1.8pt);
\fill [color=uuuuuu] (4.23,2.61) circle (1.8pt);
\fill [color=uuuuuu] (4.23,-1.7) circle (1.8pt);
\fill [color=uuuuuu] (2.99,0.46) circle (1.8pt);
\fill [color=uuuuuu] (5.48,0.46) circle (1.8pt);
\fill [color=uuuuuu] (4.23,-0.26) circle (1.8pt);
\end{scriptsize}
\end{tikzpicture}
\caption{$ $  $d=7$, $t_3=7$, Fano plane}
\label{fig: d=7,Fano}
\end{minipage}
\quad
\begin{minipage}{0.4\textwidth}
   \centering
\begin{tikzpicture}[line cap=round,line join=round,>=triangle 45,x=1.0cm,y=1.0cm,scale=0.5]
\clip(-4.5,-1.7799999999999987) rectangle (6.86,6.44);
\draw [domain=-4.5:6.86] plot(\x,{(--7.298400000000001--1.0*\x)/2.8200000000000003});
\draw [domain=-4.5:6.86] plot(\x,{(--7.4808--1.62*\x)/2.52});
\draw [domain=-4.5:6.86] plot(\x,{(--2.5488-0.9600000000000002*\x)/1.8});
\draw [domain=-4.5:6.86] plot(\x,{(--0.4032000000000002-2.58*\x)/-0.72});
\draw [domain=-4.5:6.86] plot(\x,{(-1.8660000000000005--0.6200000000000001*\x)/-0.30000000000000004});
\draw [domain=-4.5:6.86] plot(\x,{(--0.24239999999999973--1.9600000000000002*\x)/1.02});
\draw [domain=-4.5:6.86] plot(\x,{(--14.146650998618984-4.596468022286777*\x)/0.9973589218534213});
\begin{scriptsize}
\draw [fill=black] (-1.32,2.12) circle (1.5pt);
\draw [fill=black] (1.5,3.12) circle (1.5pt);
\draw [fill=black] (1.2,3.74) circle (1.5pt);
\draw [fill=black] (0.48,1.16) circle (1.5pt);
\draw [fill=black] (2.135684556407448,4.341511500547646) circle (1.5pt);
\draw [fill=black] (3.1330434782608694,-0.2549565217391311) circle (1.5pt);
\draw [fill=black] (0.9750247116968699,2.933838550247117) circle (1.5pt);
\draw [fill=black] (1.7998243045387998,5.889370424597367) circle (1.5pt);
\draw [fill=black] (2.336377952755906,3.416587926509187) circle (1.5pt);
\end{scriptsize}
\end{tikzpicture}
\caption{$ $  $d=7$, $t_3=6$, $t_2=3$}\label{fig: d=7,C}
\end{minipage}
\end{figure}
Case (v) is achievable over any, big enough, field, see Figure \ref{fig: d=7,C}.

\subsection{Eight lines}

Passing to the case of configurations of $8$ lines observe to begin with that there are at most two $4$-fold points or one $5$-fold point (see Lemma \ref{lem: multiplicity}). Hence from the combinatorial equality \eqnref{eq: combinatorial} there are the following candidates for the minimal Harbourne constant:
\begin{enumerate}
\item[\rm(i)] $T=(0,6,0,1,0,0,0)$ and $q(T)=-2\frac{1}{7}$,
\item[\rm(ii)] $T=(1,5,2,0,0,0,0)$ and $q(T)=-2,125$,
\item[\rm(iii)] $T=(1,9,0,0,0,0,0)$ and $q(T)=-2,1$,
\item[\rm(iv)] $T=(1,4,0,0,1,0,0)$ and $q(T)=-2$,
\item[\rm(v)] $T=(3,5,0,1,0,0,0)$ and $q(T)=-2$,
\item[\rm(vi)] $T=(4,4,2,0,0,0,0)$ and $q(T)=-2$,
\item[\rm(vii)] $T=(4,6,1,0,0,0,0)$ and $q(T)=-2$,
\item[\rm(viii)] $T=(4,8,0,0,0,0,0)$ and $q(T)=-2$.
\end{enumerate}
Cases (i), (iii), (v) are excluded since Lemma \ref{lem: points on one line} yields that on each configuration line there  must be an odd number of points with even multiplicity. In the first case we do not have any points with even multiplicity, in the second and the third one we do not have enough even points.\\
For the cases (ii) and (vi) observe first that the two $4$-fold points must lie on the same configuration line otherwise there cannot be any triple points by Lemma \ref{SPC}, so the double point lies on this line too. Hence we have the following situation:
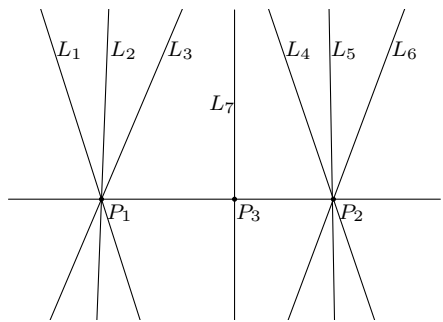
\begin{figure}[H]
\centering
\begin{tikzpicture}[line cap=round,line join=round,>=triangle 45,x=1.0cm,y=1.0cm,scale=0.5]
\clip(-4.126666666666669,-0.6400000000000037) rectangle (7.233333333333335,7.579999999999999);
\draw [domain=-4.126666666666669:7.233333333333335] plot(\x,{(--15.616-0.0*\x)/6.1});
\draw [domain=-4.126666666666669:7.233333333333335] plot(\x,{(--0.7119999999999997--0.8199999999999998*\x)/-0.26});
\draw [domain=-4.126666666666669:7.233333333333335] plot(\x,{(--1.563733333333329--0.8799999999999999*\x)/0.03333333333333188});
\draw [domain=-4.126666666666669:7.233333333333335] plot(\x,{(--2.7045333333333295--0.98*\x)/0.413333333333332});
\draw [domain=-4.126666666666669:7.233333333333335] plot(\x,{(-7.294266666666664--1.38*\x)/-0.4666666666666659});
\draw [domain=-4.126666666666669:7.233333333333335] plot(\x,{(-4.231866666666663--1.2199999999999998*\x)/0.45333333333333403});
\draw [domain=-4.126666666666669:7.233333333333335] plot(\x,{(-4.260266666666665--0.96*\x)/-0.0166666666666659324});
\draw (1.8200000000000018,-0.6400000000000037) -- (1.8200000000000018,7.579999999999999);
\begin{scriptsize}
\draw [fill=black] (-1.68,2.56) circle (1.5pt);
\draw[color=black] (-1.2,2.2) node {$P_1$};
\draw [fill=black] (4.42,2.56) circle (1.5pt);
\draw[color=black] (4.9,2.2) node {$P_2$};
\draw[color=black] (-2.526666666666668,6.5) node {$L_1$};
\draw[color=black] (-1.1,6.5) node {$L_2$};
\draw[color=black] (0.4,6.5) node {$L_3$};
\draw[color=black] (3.5,6.5) node {$L_4$};
\draw[color=black] (6.3,6.5) node {$L_6$};
\draw[color=black] (4.7,6.5) node {$L_5$};
\draw [fill=black] (1.8200000000000005,2.56) circle (1.5pt);
\draw[color=black] (2.2,2.2) node {$P_3$};
\draw[color=black] (1.5,5.079999999999998) node {$L_7$};
\end{scriptsize}
\end{tikzpicture}
\end{figure}
Lines $L_1,\ldots,L_6$ intersect in $9$ double points and only at most three of them are collinear. So on the line $L_7$ there could be at most three triple points which is a contradiction.

Lemma \ref{lem: multiplicity} yields that (iv) is impossible over any field.

Case (vii) is achievable over any, big enough, field, see Figure \ref{fig: d=8,C}.

\begin{figure}[H]
\centering
\begin{minipage}{0.4\textwidth}
\centering
\begin{tikzpicture}[line cap=round,line join=round,>=triangle 45,x=1.0cm,y=1.0cm,scale=0.3]
\clip(-7.130000000000002,-4.069999999999999) rectangle (9.910000000000002,8.260000000000003);
\draw [domain=-7.130000000000002:9.910000000000002] plot(\x,{(--23.014000000000003--2.4499999999999997*\x)/5.550000000000001});
\draw [domain=-7.130000000000002:9.910000000000002] plot(\x,{(--11.090800000000002-0.10999999999999988*\x)/4.91});
\draw [domain=-7.130000000000002:9.910000000000002] plot(\x,{(--11.0982-1.3900000000000001*\x)/7.130000000000001});
\draw [domain=-7.130000000000002:9.910000000000002] plot(\x,{(--1.46-3.25*\x)/6.25});
\draw [domain=-7.130000000000002:9.910000000000002] plot(\x,{(-0.7167999999999997--2.5599999999999996*\x)/0.64});
\draw [domain=-7.130000000000002:9.910000000000002] plot(\x,{(-5.639200000000001--3.14*\x)/-1.3400000000000003});
\draw [domain=-7.130000000000002:9.910000000000002] plot(\x,{(--11.796000000000001-5.7*\x)/0.7000000000000002});
\draw [domain=-7.130000000000002:9.910000000000002] plot(\x,{(-0.5244450588649355--2.140683549842573*\x)/1.4975499630607865});
\begin{scriptsize}
\draw [fill=black] (-4.07,2.35) circle (2.7pt);
\draw [fill=black] (1.48,4.8) circle (2.7pt);
\draw [fill=black] (0.84,2.24) circle (2.7pt);
\draw [fill=black] (2.18,-0.9) circle (2.7pt);
\draw [fill=black] (0.6380407890337678,1.4321631561350716) circle (2.7pt);
\draw [fill=black] (0.29946902654867247,0.07787610619469029) circle (2.7pt);
\draw [fill=black] (0.022153549010139856,4.156446161274746) circle (2.7pt);
\draw [fill=black] (1.7970189896094588,2.218559656037263) circle (2.7pt);
\draw [fill=black] (1.92439094484219,1.1813880205707368) circle (2.7pt);
\draw [fill=black] (2.18,-0.9) circle (2.7pt);
\draw [fill=black] (1.2343565108937131,1.315910862532642) circle (2.7pt);
\draw [fill=black] (1.48,4.8) circle (2.7pt);
\draw [fill=uuuuuu] (4.55141403865717,6.1558494404883) circle (2.7pt);
\end{scriptsize}
\end{tikzpicture}
\caption{$ $  $d=8$, $t_4=1$, $t_3=6$, $t_2=4$}\label{fig: d=8,C}
\end{minipage}
\quad
\begin{minipage}{0.4\textwidth}
   \centering
\begin{tikzpicture}[line cap=round,line join=round,x=1.0cm,y=1.0cm,scale=0.7]
\clip(0.02,-2.74) rectangle (8,3.34);
\draw [color=black] (1.74,-1.7)-- (6.72,-1.7);
\draw [color=black] (6.72,-1.7)-- (4.23,2.61);
\draw [color=black] (4.23,2.61)-- (1.74,-1.7);
\draw (1.74,-1.7)-- (5.48,0.46);
\draw (2.99,0.46)-- (6.72,-1.7);
\draw (2.99,0.46)-- (4.23,-1.7);
\draw (4.23,-1.7)-- (5.48,0.46);
\draw [shift={(4.23,0.94)}] plot[domain=-1.19:4.33,variable=\t]({1*1.68*cos(\t r)+0*1.68*sin(\t r)},{0*1.68*cos(\t r)+1*1.68*sin(\t r)});
\begin{scriptsize}
\fill [color=black] (1.74,-1.7) circle (1.5pt);
\fill [color=black] (6.72,-1.7) circle (1.5pt);
\fill [color=black] (4.23,2.61) circle (1.5pt);
\fill [color=black] (4.23,-1.7) circle (1.5pt);
\fill [color=black] (2.99,0.46) circle (1.5pt);
\fill [color=black] (5.48,0.46) circle (1.5pt);
\fill [color=black] (3.61,-0.62) circle (1.5pt);
\fill [color=black] (4.85,-0.62) circle (1.5pt);
\end{scriptsize}
\end{tikzpicture}
\caption{$ $  $d=8$, $t_3=8$, $t_2=4$}
\label{fig: d=8}
\end{minipage}
\end{figure}
For case (viii) there is the M\"obius-Kantor $(8_3)$ configuration. This
   configuration cannot be drawn in the real plane. Collinearity
   is indicated by segments and the circle arch, see Figure \ref{fig: d=8}. So the minimal value of the absolute and the complex Harbourne constant for $d=8$ is $-2$.

\subsection{Nine lines}

In the case of $d=9$ observe first that $t_6\leq 1$. Indeed, for two points with multiplicity $6$ we need at least eleven lines, which is contradiction. By a similar argument we may show that if $4$ is the highest multiplicity of points from $\mathcal{P}$, then $t_4\leq 3$. As above, for four points with multiplicity $4$ we need at least ten lines (see Lemma \ref{lem: multiplicity}).
Therefore, taking into account the above facts and from easy combinatorial calculations we obtain the following candidates for the lowest value of the Harbourne constant
\begin{enumerate}
\item[\rm(i)] $T=(0,6,3,0,0,0,0,0)$ and $q(T)=-2\frac{1}{3}$,
\item[\rm(ii)] $T=(0,8,2,0,0,0,0,0)$ and $q(T)=-2,3$,
\item[\rm(iii)] $T=(0,10,1,0,0,0,0,0)$ and $q(T)=-2\frac{3}{11}$,
\item[\rm(iv)] $T=(0,7,0,0,1,0,0,0)$ and $q(T)=-2,25$,
\item[\rm(v)] $T=(0,12,0,0,0,0,0,0)$ and $q(T)=-2,25$.
\end{enumerate}
Cases (i), (ii), (iii) and (iv) are excluded by Lemma \ref{lem: points on one line}, since every configuration line meets $8$ lines so there is the even number of even multiplicity points on each line. In the cases (iii) and (iv) we have only one point of even multiplicity. For the cases (i) and (ii) observe that two of $4$-fold points have to lie on a configuration line. Then we have the following situation:
\begin{figure}[H]
\centering
\begin{tikzpicture}[line cap=round,line join=round,>=triangle 45,x=1.0cm,y=1.0cm,scale=0.5]
\clip(-4.3,-1.9199999999999986) rectangle (7.0600000000000005,6.3);
\draw [domain=-4.3:7.0600000000000005] plot(\x,{(--10.244800000000001-0.9200000000000002*\x)/5.28});
\draw [domain=-4.3:7.0600000000000005] plot(\x,{(--0.5119999999999996--0.48*\x)/-0.14000000000000012});
\draw [domain=-4.3:7.0600000000000005] plot(\x,{(--1.3976000000000002--0.5*\x)/0.24});
\draw [domain=-4.3:7.0600000000000005] plot(\x,{(-0.03280000000000083--0.5799999999999996*\x)/-0.4600000000000002});
\draw [domain=-4.3:7.0600000000000005] plot(\x,{(-2.0991999999999997--0.5599999999999998*\x)/-0.08000000000000007});
\draw [domain=-4.3:7.0600000000000005] plot(\x,{(-5.9552--1.4799999999999998*\x)/-0.52});
\draw [domain=-4.3:7.0600000000000005] plot(\x,{(-3.8919999999999986--1.16*\x)/0.18000000000000016});
\begin{scriptsize}
\draw [fill=black] (-1.72,2.24) circle (1.5pt);
\draw [fill=black] (3.56,1.32) circle (1.5pt);
\draw[color=black] (-4,2.2) node {$L_1$};
\draw[color=black] (-2.1,4.82) node {$L_3$};
\draw[color=black] (-0.16,4.6) node {$L_4$};
\draw[color=black] (-3.68,4) node {$L_2$};
\draw[color=black] (2.6,5.7) node {$L_6$};
\draw[color=black] (1.86,5) node {$L_5$};
\draw[color=black] (3.8,5.38) node {$L_7$};
\end{scriptsize}
\end{tikzpicture}
\caption{$ $  {Line with two $4$-fold points}}\label{fig: Line with two $4$-fold points}
\end{figure}
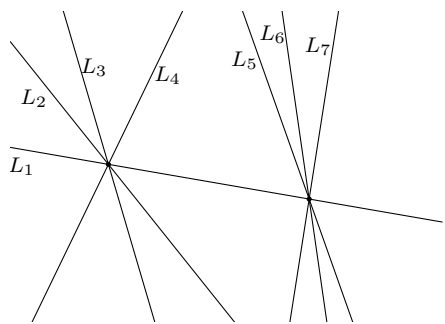
The third $4$-fold point has to lie on each line $L_2,\ldots,L_7$, which is a contradiction and excludes the case (i). In the case (ii) there is only one point with even multiplicity on lines $L_2,\ldots,L_7$.

For $d=9$ there is the dual Hesse configuration. It is easier to describe the original Hesse configuration. It arises taking the nine
order $3$ torsion points of a smooth complex cubic curve (which carries the structure of an abelian group and the torsion is understood with respect to this group structure). There are $12$ lines passing through the nine points in such a way that each line contains exactly $3$ torsion points and there are $4$ lines passing through each of the points. See \cite{ArtDol09} for details. This configuration cannot be drawn in the real plane.\\

Beside the geometrical realization over the complex numbers, the same configuration can be also easily obtained in characteristic $3$, more precisely, in the plane $\P^2(\F_3)$ taking all $13$ lines and removing from this set all $4$ lines passing through a fixed point. We leave the details to the reader.

\subsection{Ten lines}

 Finally we pass to the most involved case. To begin with we observe 
 that adding a line passing through two triple points to the dual Hesse configuration
 yields a configuration of $10$ lines with $t_4=2$, $t_3=10$, $t_2=3$ and $H_{L}(\mathbb{K},\mathcal{L})=-2\frac{4}{15}\simeq -2,2667$. It suffices thus to consider solutions $T$ of the combinatorial equality \eqnref{eq: combinatorial} 
 with quotients less than $-2\frac{4}{15}$. This list is quite long. However a considerable part of it
 can be excluded by the Triangular Inequality (Lemma \ref{lem: multiplicity}), the Quadrangle Inequality (Lemma \ref{lem: multiplicity4}) and the Parity Criterion (Lemma \ref{lem: points on one line}). In the end we are left with the following candidates for the absolute linear Harbourne constant of degree $10$:
 \begin{center}
   \renewcommand{\arraystretch}{1.2}
   \begin{tabular}{|p{0.8cm}|p{0.7cm}|p{0.7cm}|p{0.7cm}|p{0.7cm}|p{0.7cm}|p{3cm}|}
   \hline
   T & $t_2$ & $t_3$ & $t_4$ & $t_5$ & $t_6$ & $q(T)$ \\
   \hline
   (i) &  & $1$ & $7$ &  &  & -2,625 \\
   \hline
   (ii) &  & $3$ & $6$ &  &  & -2,5555555556 \\
   \hline
   (iii) &  & $5$ & $5$ &  &  & -2,5 \\
   \hline
   (iv) &  & $7$ & $4$ &  &  & -2,4545454545 \\
   \hline
   (v) &  & $8$ & $1$ &  & $1$ & -2,4 \\
   \hline
   (vi) & $3$ &  & $7$ &  &  & -2,4 \\
   \hline
   (vii) & $2$ & $7$ & $2$ & $1$ &  & -2,3333333333 \\
   \hline
   (viii) &  & $9$ & $3$ &  &  & -2,4166666667 \\
   \hline
   (ix) & $3$ & $2$ & $6$ &  &  & -2,3636363636 \\
   \hline
   (x) & $3$ & $4$ & $5$ &  &  & -2,3333333333 \\
   \hline
   (xi) & $3$ & $6$ & $4$ &  &  & -2,3076923077 \\
   \hline
   (xii) & $3$ & $8$ & $3$ &  &  & -2,2857142857 \\
   \hline
   \end{tabular}
 \end{center}
Now we consider cases (i), (ii), (iii), (iv). Since there are at least three $4$-fold points, there must be a line passing through two of them. The Parity Lemma (Lemma \ref{lem: points on one line}) implies then that this line contains also a third $4$-fold point, since there are no double points in these configurations. Hence we have the situation indicated on the following picture
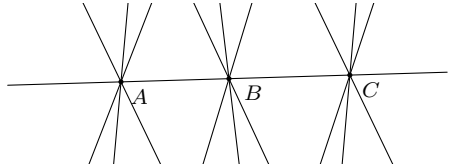
\begin{figure}[H]
\centering
\begin{tikzpicture}[line cap=round,line join=round,>=triangle 45,x=1.0cm,y=1.0cm,scale=0.5]
\clip(-4.260000000000001,1.5400000000000003) rectangle (7.300000000000002,5.7799999999999985);
\draw [domain=-4.260000000000001:7.300000000000002] plot(\x,{(--22.5044--0.17999999999999972*\x)/6.0200000000000005});
\draw [domain=-4.260000000000001:7.300000000000002] plot(\x,{(-0.6299999999999999--1.2999999999999998*\x)/-0.6199999999999999});
\draw [domain=-4.260000000000001:7.300000000000002] plot(\x,{(--2.540400000000001--1.58*\x)/0.14000000000000012});
\draw [domain=-4.260000000000001:7.300000000000002] plot(\x,{(-3.790000000000001-1.4000000000000004*\x)/-0.54});
\draw [domain=-4.260000000000001:7.300000000000002] plot(\x,{(-5.6784--1.6399999999999997*\x)/0.54});
\draw [domain=-4.260000000000001:7.300000000000002] plot(\x,{(-8.834800000000001--1.3399999999999999*\x)/-0.6400000000000006});
\draw [domain=-4.260000000000001:7.300000000000002] plot(\x,{(-6.170400000000001--1.4000000000000004*\x)/0.1200000000000001});
\draw [domain=-4.260000000000001:7.300000000000002] plot(\x,{(-1.3562021644174997--0.6750811656561706*\x)/-0.08006323749918542});
\draw [domain=-4.260000000000001:7.300000000000002] plot(\x,{(-0.5869882782449976--1.3950811656561704*\x)/0.4199367625008146});
\draw [domain=-4.260000000000001:7.300000000000002] plot(\x,{(-4.015397353152098--1.2150811656561706*\x)/-0.5600632374991854});
\begin{scriptsize}
\draw [fill=black] (-1.28,3.7) circle (1.5pt);
\draw[color=black] (-0.8,3.3) node {$A$};
\draw [fill=black] (4.74,3.88) circle (1.5pt);
\draw[color=black] (5.3,3.5) node {$C$};
\draw [fill=black] (1.5600632374991854,3.7849188343438294) circle (1.5pt);
\draw[color=black] (2.2,3.4) node {$B$};
\end{scriptsize}
\end{tikzpicture}
\end{figure}
Since there are all configuration lines already visible in that picture, it is clear that there is no way to produce a fourth $4$-fold point out of lines in $3$ pencils $A$, $B$ and $C$ (see The Second Pencil Criterion). Therefore cases (i), (ii), (iii) and (iv) are excluded.\\
In the case (v) note that the Parity Criterion yields that all configuration lines belong either to the pencil through the $4$-fold point or to the pencil through the $6$-fold point (the common line cannot be a configuration line). By the Second Pencil Criterion (Lemma \ref{SPC}) there is no way to obtain any triple points.\\
For the case (vi) observe that there exists a line $L_1$ with at least two $4$-fold points (as in the Figure~\ref{fig: Line with two $4$-fold points}). Lines $L_2,\ldots,L_7$ intersect $L_1$ in $9$ double points. Together with the two $4$-fold points there would be already $11$ configuration points, a contradiction.\\
Similarly case (vii) is excluded by the First Pencil Criterion (Lemma \ref{boundary_of_points}), since the $5$-fold point and $4$-fold point must lie on a configuration line.\\
For the last five possibilities (e.g. (viii), (ix), (x), (xi) and (xii)) note that Lemma \ref{hirzebruch} implies that these configurations do not exist in the complex projective plane $\mathbb{P}^2$.\\
The configuration (viii) does however exist over some other fields. Indeed if we take all lines and all points in the plane $\mathbb{P}^2(\mathbb{F}_3)$ then we have a configuration of thirteen lines and thirteen $4$-fold points. Then we remove from this set $3$ lines passing through a fixed point. This gives a configuration of ten lines, three $4$-fold points and nine triple points. Thus the absolute linear Harbourne constant of configuration of degree $10$ is equal to $-2\frac{5}{12}$ and it is achieved in the plane $\mathbb{P}^2(\mathbb{F}_3)$.
\end{proof}

\paragraph*{\emph{Acknowledgement.}}
I would like to thank T. Szemberg for helpful discussions.



\bigskip \small

\bigskip
   Justyna Szpond,
   Instytut Matematyki UP,
   Podchor\c a\.zych 2,
   PL-30-084 Krak\'ow, Poland

\nopagebreak



   \textit{E-mail address:} \texttt{szpond@up.krakow.pl}


\end{document}